\documentclass[11pt]{amsart}
\usepackage{amsmath}
\usepackage{latexsym, bm}
\usepackage{indentfirst}
\usepackage{geometry}
\usepackage{array}
\usepackage{lipsum}
\usepackage{setspace}
\usepackage{multirow}
\usepackage{tabu}
\usepackage{amsfonts}
\usepackage{amsthm}
\usepackage{fancyhdr}
\usepackage{tikz}
\usepackage{amscd}
\usepackage{amssymb}
\usepackage{enumerate}
\usepackage{amssymb,amsfonts}
\usepackage[all,arc]{xy}
\usepackage{enumerate}
\usepackage{mathrsfs}
\usepackage{CJK}
\usepackage{hyperref}
\usepackage{cite}
\usepackage{tikz-cd}
\usepackage{extarrows}

\usepackage{fancyhdr}

\def \f{\frac}
\def \a{\alpha}
\def \b{\bar}

\def \lra{\longrightarrow}

\def \p{\partial}
\def \vphi{\varphi}

\def \om{\omega}

\def \g{\gamma}
\def \d{\delta}
\def \t{\tilde}
\def \B{\mathbb}
\def \ov{\overline}

\newtheorem{thm}{Theorem}[section]
\newtheorem{cor}[thm]{Corollary}
\newtheorem{lem}[thm]{Lemma}
\newtheorem{prop}[thm]{Proposition}

\newtheorem{defn}[thm]{Definition}
\newtheorem{remk}[thm]{Remark}

\newtheorem{que}[thm]{Question}

\begin{document}
\title{Liouville Theorems for holomorphic maps on pseudo-Hermitian manifolds}
\author{Haojie Chen }
\address{Department of Mathematics\\ Zhejiang Normal University\\ Jinhua Zhejiang, 321004, China}
\email{chj@zjnu.edu.cn}
\author{Yibin Ren}
\address{Department of Mathematics\\ Zhejiang Normal University\\ Jinhua Zhejiang, 321004, China}
\email{allenryb@outlook.com}
\maketitle

\begin{abstract} We prove some Liouville type results for generalized holomorphic maps in three classes: maps from pseudo-Hermitian manifolds to almost Hermitian manifolds, maps from almost Hermitian manifolds to pseudo-Hermitian manifolds and maps from pseudo-Hermitian manifolds to pseudo-Hermitian manifolds, assuming that the domains are compact. For instance, we show that any $(J,J^N)$ holomorphic map from a compact pseudo-Hermitian manifold $M$ with nonnegative (resp. positive) pseudo-Hermitian sectional curvature to an almost Hermitian manifold $N$ with negative (resp. nonpositive) holomorphic sectional curvature is constant. We also construct explicit almost CR structures on a complex vector bundle over an almost CR manifold.
\end{abstract}

\section{Introduction}

The Liouville theorem is a landmark result in complex analysis. Its generalizations to higher dimension and to other setting are crucial and are closely related to the study of Schwarz lemma. Recall the classical Schwarz-Pick lemma states that every holomorphic maps from the unit disc to itself decreases the Poincar\'e metric. It has been extended to more general settings by the work of Ahlfors, Chern, Yau and others (\cite{A},\cite{CCL},\cite{C},\cite{Roy},\cite{Yau}). In particular, Yau's Schwarz lemma \cite{Yau} says that every holomorphic map from a complete K\"ahler manifold $M$ with Ricci curvature bounded below by a constant $-K_1\leq 0$ into a Hermitian manifold $N$ with holomorphic bisectional curvature bounded above by a constant $-K_2<0$ is distance decreasing up to a constant $\frac{K_1}{K_2}$. If $K_1=0$, then the Liouville type theorem that any holomorphic map from $M$ to $N$ is constant is derived. Later, Royden \cite{Roy} improved Yau's result by only requiring $N$ has holomorphic sectional curvature bounded above by a negative constant if $N$ is K\"ahler. Since then, there have been a lot of developments on the generalizations and applications of Schwarz lemma. We refer to \cite{Ko}\cite{To}\cite{TY}\cite{WY} and especially to \cite{Ni1}\cite{Ni2} for the recent advances.

 Recently, using a new method, Yang \cite{Yang1} proves a Liouville type theorem on Hermitian manifolds as follows  (see previous works in \cite{Roy}, \cite{YC}): any holomorphic map from a compact Hermitian manifold with nonnegative (resp. positive) holomorphic sectional curvature to a Hermitian manifolds with negative (resp. nonpositive) holomorphic sectional curvature must be constant. Yang studies some new energy density functions which have several  other applications \cite{Yang1} . The above theorem is then generalized to almost Hermitian setting by Masood \cite{Mas}, using a similar approach.

The main purpose of this paper is to derive some Liouville type theorems in the category of pseudo-Hermitian manifolds. Recall that a pseudo-Hermitian manifold is a $2m_0+1$ dimensional manifold $M$ with a codimension one CR structure $(HM,J)$ and a one-form $\theta$ such that
$HM=\ker (\theta)$ and the Levi form $L_\theta(X,Y)=d\theta(X,JY)$ is positively definite. It is naturally related to the study of CR geometry and several complex variables.

 There has been certain progress on the study of Schwarz and Liouville type results on pseudo-Hermitian manifolds recently (\cite{CDRY}\cite{DRY}). In particular, In \cite{DRY}, Dong-Ren-Yu derive the Schwarz lemmas for CR maps and transversally holomorphic maps between pseudo-Hermitian manifolds. In \cite{CDRY}, Chong-Dong-Ren-Yu study Schwarz lemmas for generalized holomorphic maps between pseudo-Hermitian manifolds and Hermitian manifolds. As applications, they show that any $(J,J^N)$-holomorphic map (resp. $(J^N,J)$-holomorphic map) from a complete pseudo-Hermitian manifold with nonnegative pseudo-Hermitian Ricci curvature (resp. from a complete K\"ahler manifold with nonnegative Ricci curvature) to a Hermitian manifold with negative holomorphic bisectional curvature (resp. to a Sasakian manifold with negative pseudo-Hermitian bisectional curvature) is constant (resp. horizontally constant).

By investigating generalizations of Yang's energy function and introducing local normal quasi holomorphic frames on pseudo-Hermitian manifolds, we prove the following result.

\begin{thm}
Let $(M,HM,J,\theta)$ be a compact pseudo-Hermitian manifold with non-negative (resp. positive) pseudo-Hermitian sectional curvature and $(N,J^N,h)$ be an almost Hermitian manifold with negative (resp. non-positive) holomorphic sectional curvature. Then any $(J,J^N)$ holomorphic map from $M$ to $N$ is constant.
\end{thm}

Here, a $(J,J^N)$ holomorphic map is an almost CR map such that the image of the Reeb vector field under the tangent map is zero (see \cite{CDRY} or Definition \ref{def1} (i) in section 2).

Next, we consider a generalized holomorphic map from an almost Hermitian manifold to a pseudo-Hermitian manifold (Definition \ref{def1} (ii)). Recall that a pseudo-Hermitian manifold is called Sasakian if its pseudo-Hermitian torsion one form vanishes. We prove
\begin{thm}
Let $(N,J^N,h)$ be a compact almost Hermitian manifold with non-negative (resp. positive) holomorphic sectional curvature and $(M,HM,J,\theta)$ be a Sasakian manifold with negative (resp. nonpositive) pseudo-Hermitian sectional curvature. Then any $(J^N,J)$ holomorphic map from $N$ to $M$ is horizontally constant.
\end{thm}

For holomorphic map between pseudo-Hermitian manifolds (see Definition \ref{def2}), we
prove the following.
\begin{thm}
Let $(M,HM,J,\theta)$ be a compact pseudo-Hermitian manifold with non-negative (resp. positive) pseudo-Hermitian sectional curvature and $(N,HN, J^N, \theta^N)$ be a Sasakian manifold with negative (resp. non-positive) pseudo-Hermitian sectional curvature. Then any transversally holomorphic map from $M$ to $N$ is horizontally constant.
\end{thm}

We make use of normal quasi holomorphic frames on pseudo-Hermitian manifolds in the proofs of the above results. This notion is first introduced by Yu \cite{Yu1} on almost complex manifolds and has nice applications in \cite{Mas}. We introduce it on almost CR manifolds and show the existence. It has some similar properties with the local normal frame on K\"ahler manifolds. For example, with respect to the normal quasi holomorphic frame, the Christoffel symbols of the Tanaka-Webster connection and their partial differentials vanish at a point, i.e. $\Gamma_{\b{i}j}^k=\Gamma_{ij}^k=0, \p\Gamma_{\b{i}j}^k=\b{\p}\Gamma_{\b{i}j}^k=0$.

In the last section, we study almost CR structures on a complex vector bundle over an almost CR manifold. It is well known that a holomorphic vector bundle over a complex manifold is itself a complex manifold such that the projection is holomorphic. We prove an analogue on almost CR manifolds (Proposition 5.1).

\begin{prop} \label{prop1.4} Let $p: E\lra M$ be a $k$-dimensional complex vector bundle over $M$ with a codimension $m-2m_0$ almost CR structure. Then associated to any linear connection on $E$, there is a codimension $m-2m_0$ almost CR structure $(HE, J_E)$ on $E$ such that the map $p$ is an almost CR map.
\end{prop}

If $m=2m_0$, an almost CR structure is just an almost complex structure. In this case, the existence of almost complex structure on a complex vector bundle has been well studied (\cite{DT},\cite{Kru},\cite{LS},\cite{YI},\cite{CZ}). It is natural to ask when the almost CR structure $(HE, J_E)$ in Proposition \ref{prop1.4} is integrable. When $M$ is a CR manifold, we can construct a CR structure on its tangent bundle $TM$. See section 5 for more details and discussion.

The structure of the paper is as follows. In section 2, we give some preliminary background and notations. In section 3, we introduce the normal quasi holomorphic local frames on almost CR manifolds and pseudo-Hermitian manifolds and show the existence. We also discuss a maximum principle on an almost CR manifold. In section 4, we give the proof of the main theorems. In section 5, we discuss the construction of almost CR structures on a complex vector bundle over an almost CR manifold.\\

\noindent \textbf{Acknowlegements.} The authors would like to thank Professors Yuxin Dong, Jiaping Wang and Fangyang Zheng for very helpful suggestions and comments. This research work is partially supported by NSFC grant No. 11901530, 11801517 and Zhejiang
Provincial NSF grant No. LY19A010017. 

\section{Notations}
In this section, we introduce some notations in pseudo-Hermitian geometry (cf. \cite{DT}). Let $M$ be a $m$-dimensional orientable manifold. Denote the tangent bundle of $M$ by $TM$. For any integer $0\leq m_0\leq m/2$, a codimension $m-2m_0$ almost CR structure on $M$ is a complex subbundle $T_{1,0}M$ of complex rank $m_0$ in $TM\otimes \mathbb C$ such that $T_{1,0}M\cap T_{0,1}M=\{0\}$ where $T_{0,1}M=\overline{T_{1,0}M}$. An almost CR structure is called a CR structure if $[\Gamma(T_{1,0}M),\Gamma(T_{1,0}M)]\subseteq \Gamma(T_{1,0}M)$, where $\Gamma(T_{1,0}M)$ are all smooth sections of $T_{1,0}M$.

Denote $HM=Re\{T_{0,1}M\oplus T_{0,1}M\}$ which is called the Levi distribution. There is a natural endomorphism $J$ on $HM$ defined by $J(X+\bar{X})=i(X-\bar{X})$. Then $J^2=-id$. Equivalently, an almost CR structure can be described as a $2m_0$ dimensional subbundle $HM$ in $TM$ with a bundle map $J: HM\lra HM$ satisfying $J^2=-id$. $(HM, J)$ is a CR structure if and only if the following integrability conditions hold (see \cite{DT}):
\begin{itemize}
\item $[JX,Y]+[X,JY] \in \Gamma(M, HM)$,
\item $[JX,JY]-J[JX,Y]-J[X,JY]-[X,Y]=0.$
\end{itemize}
where $X, Y$ are smooth sections of $HM$. When $m=2m_0$, an almost CR structure is just an almost complex structure and a CR structure is a complex structure by the Newlander-Nirenberg theorem. Assume that $(N,HN,J^N)$ is another almost CR manifold. A map $f:M\lra N$ is called an almost CR map if \begin{align}\label{almost CR}
df(HM)\subset HN, \ \ \ \ \  df\circ J=J^N\circ df\ \ \ \text{on}  \ \ HM.
\end{align}

Consider a codimension one CR structure when $m=2m_0+1$. As $M$ is orientable, it follows that there exists a global nowhere vanishing 1-form $\theta$ such that $HM=\ker(\theta)$. The Levi form $L_{\theta}$ is defined as $L_{\theta}(X,Y)=d\theta(X,JY)$ for $X,Y\in HM$. If $(HM, J)$ is a CR structure, then the integrability conditions imply that $L_\theta$ is $J$-invariant and symmetric.

\begin{defn}
A codimension one CR manifold $(M,HM,J)$ is called strictly convex if there exists a 1-form $\theta$ such that $HM=\ker(\theta)$ and $L_\theta$ is positive definite. In this case, $(M,HM,J,\theta)$ is called a pseudo-Hermitian manifold.
\end{defn}

For a pseudo-Hermitian manifold $(M,HM,J,\theta)$, there is a nowhere zero Reeb vector field $\xi$ defined by $\theta(\xi)=1$ and $i_{\xi}d\theta=0$. Then $TM=HM\oplus L$, where $L$ is the line bundle generated by $\xi$. Extend $J$ to $TM$ by letting $J(\xi)=0$. Denote $\pi_H:TM\lra HM$ the projection map and define $G_{\theta}(X,Y)=L_{\theta}(\pi_HX,\pi_HY)$ for $X,Y\in TM$. Finally define $$g_{\theta}=G_{\theta}+\theta\otimes \theta.$$
By the positivity and $J$-invariance of $L_\theta$, $g_\theta$ is a $J$-invariant Riemannian metric which is called the Webster metric. There is a canonical connection preserving the CR structure and the Webster metric, called the Tanaka-Webster connection.

\begin{defn}
Let $(M,HM,J,\theta)$ be a pseudo-Hermitian manifold with the Reeb vector field $\xi$ and the Webster metric $g_\theta$. Then there is a unique linear connection $\nabla$ on $M$ satisfying the following:\\
(i) $HM$ is parallel with respect to $\nabla$;\\
(ii) $\nabla J=0$ and $\nabla g_\theta=0$;\\
(iii) The torsion $T_\nabla$ satisfies: $$T_\nabla(X,Y)=2d\theta(X,Y)\xi,\ \  T_\nabla(\xi,JX)+JT_\nabla(\xi,X)=0$$ for any $X,Y\in HM$.
\end{defn}
The pseudo-Hermitian torsion $\tau$ is then a $TM$-valued 1-form defined by $\tau(X)=T_\nabla(\xi,X)$ for any $X\in TM$. A pseudo-Hermitian manifold is called a \textbf{Sasakian} manifold if $\tau\equiv 0$.

Associated to a pseudo-Hermitian manifold, there are three classes of generalized holomorphic maps which we will study. They are: \begin{itemize}
\item[(i)] holomorphic maps from pseudo-Hermitian manifolds to almost Hermitian manifolds,
\item[(ii)] holomorphic maps from almost Hermitian manifolds to pseudo-Hermitian manifolds
\item[(iii)] holomorphic maps from pseudo-Hermitian manifolds to pseudo-Hermitian manifolds.
\end{itemize}

 Assume that $(N,J^N)$ is an almost complex manifold. The following definitions are natural extensions of holomorphic maps between complex manifolds. (cf. \cite{CDRY},\cite{DRY}).

\begin{defn}\label{def1} (i) A map $f:M\lra N$ is called $(J,J^N)$ holomorphic if on $TM$ it holds $$df\circ J=J^N\circ df, $$
(ii) A map $f:N\lra M$ is called $(J^N,J)$ holomorphic if on $TN$ it holds
$$df_H\circ J^N=J\circ df_H$$ where $df_H=\pi_H\circ df$ with $\pi_H:TM\lra HM$ being the projection map.
\end{defn}

In (i), as $J(\xi)=0$, by (\ref{almost CR}), $f$ is $(J,J^N)$ holomorphic if and only if $f$ is an almost CR map and satisfies $df(\xi)=0$. In (ii), however, $f$ is $(J^N,J)$ holomorphic and almost CR if and only if $f$ is constant, as shown in \cite{CDRY} (Remark 2 there).

Finally, assume that $(N,HN, J^N, \theta^N)$ is a pseudo-Hermitian manifold with $TN=HN\oplus L^N$.
\begin{defn} \label{def2}
A map $f:M\lra N$ is called $(J,J^N)$ holomorphic if on $HM$ it satisfies $$df_{M,N}\circ J=J^N\circ df_{M,N}$$ where $df_{M,N}=\pi_{HN}\circ df\circ i_{HM}$ with $i_{HM}:HM\lra TM$ being the inclusion map.
 $f$ is called \textbf{transversally holomorphic} if it is $(J,J^N)$ holomorphic and satisfies $df(L)\subset L^N$. $f$ is called to be \textbf{horizontal constant} if $f(M)$ is in a leaf of the Reeb vector field of $N$.
\end{defn}

\section{Normal frames on almost CR and pseudo-Hermitian manifolds}
In this section, we introduce the notions of pseudoholomorphic and quasi holomorphic local frames on an almost CR manifold. In the case of almost complex manifolds, they are first introduced in \cite{Yu1}. We show that they exist on general almost CR manifolds.
\begin{defn}
Let $(M,HM,J)$ be an almost CR manifold and $p\in M$.
\begin{itemize}
\item[(i)] A local $(1,0)$ vector field $u$ around $p$ is said to be pseudoholomorphic at $p$ if $[u,\b{v}]_{1,0}(p)=0$ for any (1,0) vector field $v$ around $p$, where $X_{1,0}$ denotes the $(1,0)$ component of a vector $X$.
\item[(ii)] For a $(1,0)$ vector field $u$ which is pseudoholomorphic at $p$, if in addition it satisfies $[[u,\b{v}]_H, w]_{1,0}(p)=0$, for any (1,0) vector fields $v,w$ pseudoholomorphic at $p$, where $X_H$ denotes the projection of $X$ to $HM$, then $u$ is called quasi holomorphic at $p$.
\end{itemize}
\end{defn}
We prove
\begin{prop}\label{3.2}
Let $(M, HM, J)$ be a codimension $m-2m_0$ almost CR manifold and $p\in M$. Then there always exists a local (1,0)-frame $\{e_1, e_2, \cdots, e_{m_0}\}$ of $T_{1,0}M$ around $p$ such that $e_i$ is quasi holomorphic at $p$ for each $i$.
\end{prop}
\begin{proof}
The proof is similar to the case of almost complex manifolds \cite{Yu1}. We first show that there always exist a local $(1,0)$ pseudoholomorphic frame around each point. Assume that $\{v_1, \cdots, v_{m_0}\}$ is an arbitrary local $(1,0)$ frame of $T_{1,0}M$ around $p$ with $[v_i,\b{v}_j]_{1,0}=c_{i\b{j}}^kv_k$. Define $u_i=f_{i}^jv_j$ with $f_i^j(p)=\delta_i^j, \b{v}_k(f_i^j)(p)=c_{i\b{k}}^j(p)$. Then at $p$ we have $$[u_i,\b{v}_k]_{1,0}(p)=f_i^j(p)c_{j\b{k}}^l(p)v_l-\b{v}_k(f_i^j)(p)v_j=0.$$
For each $(1,0)$ vector $v=\a^k v_k$, $[u_i,\b{v}]_{1,0}(p)=\a^k[u_i,\b{v}_k]_{1,0}(p)=0$. So $u_i$ is pseudoholomorphic at $p$ for each $i$.

From $\{u_i, 1\leq i\leq m_0\}$, we could construct a quasi holomorphic frame $\{e_i,1\leq i\leq m_0\}$. Assume that $[u_i, \b{u}_j]_{1,0}=b_{i\b{j}}^ku_k$, such that $b_{i\b{j}}^k(p)=0$ due to the pseudoholomorphic property. Define $e_i=\phi_i^j u_j$ with \begin{align*}
\phi_i^j(p)=\delta_i^j, \b{u}_k(\phi_i^j)(p)=0, \ \text{and}\ \  u_k\b{u}_l(\phi_i^j)(p)=u_k(b_{i\b{l}}^j)(p).\end{align*} Then we claim that $\{e_i\}$ forms a local quasi holomorphic frame around $p$. To see this, first we have
$$[e_i,\b{u}_k]_{1,0}(p)=\phi_i^j[u_j,\b{u}_k]_{1,0}(p)-\b{u}_k(\phi_i^j)(p)u_j=0$$ for each $k$. So as before, $e_i$ is pseudoholomorphic at $p$. Then we compute
\begin{align*} &[[e_i,\b{u}_l]_H,u_k]_{1,0}(p)=[[e_i,\b{u}_l]_{1,0}+[e_i,\b{u}_l]_{0,1},u_k]_{1,0}(p)\\ &=[[e_i,\b{u}_l]_{1,0},u_k]_{1,0}(p)=[\phi_i^j b_{j\b{l}}^m u_m-\b{u}_l(\phi_i^j)u_j,u_k]_{1,0}(p)\\ &=\phi_i^j b_{j\b{l}}^m(p)[u_m,u_k]_{1,0}(p)-u_k(\phi_i^j b_{j\b{l}}^m)(p)u_m-\b{u}_l(\phi_i^j)(p)[u_j,u_k]_{1,0}(p)+u_k\b{u}_l(\phi_i^j)u_j\\
&=-\phi_i^j(p)u_k(b_{j\b{l}}^m)u_m+u_k\b{u}_l(\phi_i^j)(p)u_j=0
\end{align*} for each $k,l$. Now for each $(1,0)$ pseudoholomorphic vector fields $\a=\a^ju_j, \beta=\beta^j u_j$ around $p$, we have $\b{w}(\a^j)(p)=\b{w}(\beta^j)(p)=w(\b{\a}^j)(p)=w(\b{\beta}^j)(p)=0$ for each $(1,0)$ vector $w$ due to the pseudoholomorphic
property. Then
\begin{align*}
&[[e_i,\b{\a}]_H,\beta]_{1,0}(p)=[[e_i,\b{\a}]_{1,0}+[e_i,\b{\a}]_{0,1},\beta]_{1,0}(p)\\&=[[e_i,\b{\a}]_{1,0},\beta]_{1,0}(p)=[\b{\a}^j[e_i,\b{u}_j]_{1,0},\beta]_{1,0}(p)\\&=\b{\a}^j([e_i,\b{u}_j]_{1,0}(p))(\beta^k)u_k(p)-\beta^k u_k(\b{\a}^j)(p)[e_i,\b{u}_j]_{1,0}(p)+\b{\a}^j\beta^k[[e_i,\b{u}_j]_{1,0},u_k]_{1,0}(p)\\&=\b{\a}^j\beta^k[[e_i,\b{u}_j]_H,u_k]_{1,0}(p)=0.
\end{align*}
So $\{e_i,1\leq i\leq m_0\}$ forms a local quasi holomorphic frame around $p$ by definition.
\end{proof}

Let $(M, HM, J, \theta)$ be a pseudo-Hermitian manifold. We then get a better quasi holomorphic frame as follows.

\begin{prop} \label{3.3}
Let $(M, HM, J, \theta)$ be a pseudo-Hermitian manifold and $p\in M$. Then there always exists a local (1,0)-frame $\{e_1, e_2, \cdots, e_{m_0}\}$ around $p$ satisfying:
 \begin{itemize}
\item[(i)] $e_i$ is quasi holomorphic at $p$;
\item[(ii)] $g_{i\b{j}}(p)=\d_{i\b{j}}$ and $d^Hg_{i\b{j}}(p)=0$, where $d^H$ is the horizontal differential;
\item[(iii)] $[\xi,e_i]_{1,0}(p)=0$ for each $i$, with $\xi$ being the Reeb vector field.
\end{itemize}
\end{prop}
\begin{proof}
We follow the same construction as in Proposition \ref{3.2}, with some additional requirements. First, we choose the initial frame $\{v_i\}$ to be unitary at $p$. Using the same procedure to define $u_i$ (namely, $u_i=f_{i}^jv_j$ with $f_i^j(p)=\delta_i^j, \b{v}_k(f_i^j)(p)=c_{i\b{k}}^j(p)$) so that $u_i$ is pseudoholomorphic for each $i$. As $f_i^j(p)=\delta_i^j$, $\{u_i\}$ are also unitary at $p$. Consider the Tanaka-Webster connection $\nabla$. As $\nabla J=0$, we have $\nabla_{\b{u}_i} u_j=\Gamma_{\b{i}j}^ku_k, \nabla_{u_i} u_j=\Gamma_{ij}^ku_k$. Since $T_\nabla(\b{u}_i,u_j)=2\sqrt{-1}\delta_{ij}\xi$, we get that $\nabla_{\b{u}_i} u_j=[\b{u}_i, u_j]_{1,0}$. So $\Gamma_{\b{i}j}^k(p)=0$ for any $i,j,k$.

Assume at $p$, $[\xi,u_i]_{1,0}=t_i^ju_j$. We then define $e_i=\phi_i^ju_j$ with the previous requirements that $\phi_i^j(p)=\delta_i^j$, $\b{u}_k(\phi_i^j)(p)=0$, $u_k\b{u}_l(\phi_i^j)(p)=u_k(b_{i\b{l}}^j)(p)$ and the additional requirements: $u_k(\phi_i^j)(p)=-\Gamma_{ij}^k(p)$, $\xi(\phi_i^j)(p)=-t_i^j(p)$. This is possible since $u_k,\b{u}_k,\xi$ are all
linearly independent. Then $\{e_i\}$ is a local quasi holomorphic frame as before, satisfying (i). Also at $p$, we have $\nabla_{u_i}e_j=u_i(\phi_j^k)u_k+\phi_j^k\Gamma_{ik}^lu_l=0$ and $\nabla_{\b{u}_j}e_k=\b{u}_i(\phi_j^k)u_k+\phi_j^k\Gamma_{\b{i}k}^lu_l=0$. So $\nabla_Xe_k=\nabla_X \b{e}_k=0$ for any $X\in HM$. Then $Xg_{i\b{j}}=g(\nabla_X e_i,\b{e}_j)+g(e_i,\nabla_X \b{e}_j)=0$. Therefore, $d_Hg_{i\b{j}}=0$ satisfying (ii). For (iii), we directly get $[\xi,e_i]_{1,0}(p)=\xi(\phi_i^j)(p)u_j+\phi_i^j(p)t_j^k(p)u_k=0$. So the frame $\{e_i,1\leq i\leq m_0\}$ satisfies all the properties.
\end{proof}

\begin{defn} Let $(M, HM, J, \theta)$ be a $m=2m_0+1$ dimensional pseudo-Hermitian manifold and $p\in M$. A local frame $\{e_i,1\leq i\leq m_0\}$ is call a normal quasi holomorphic frame at $p$ if it satisfies: (i) $e_i$ is quasi holomorphic at $p$ for each i, (ii) $g_{i\b{j}}(p)=\d_{i\b{j}}$ and $d^Hg_{i\b{j}}(p)=0$, (iii) $[\xi,e_i]_{1,0}(p)=0$ for each $i$.
\end{defn}

By Proposition \ref{3.3}, there always exist local normal quasi holomorphic frames around any point on a pseudo-Hermitian manifold. Assume that $\{e_1, e_2, \cdots, e_{m_0}\}$ is a normal quasi holomorphic frame at $p\in M$. Let $\nabla$ be the Tanaka-Webster connection with $\nabla_{e_i}e_j=\Gamma_{ij}^ke_k$, $\nabla_{\b{e}_i}e_j=\Gamma_{\b{i}j}^ke_k$. Let $R$ be the curvature tensor. We have
\begin{prop}\label{normal}
Let $(M, HM, J, \theta)$ be a $2m_0+1$-dimensional pseudo-Hermitian manifold. For a normal quasi holomorphic frame at $p\in M$, the following hold at $p$:
\begin{align*}
\Gamma_{\b{i}j}^k=\Gamma_{ij}^k=0,\ \ \p\Gamma_{\b{i}j}^k=\b{\p}\Gamma_{\b{i}j}^k=0,\ \ \
R_{i\b{j}k\b{l}}&=g(R(e_i,\b{e}_j)e_k,\b{e}_l)=-\b{e}_je_i(g_{k\b{l}}).
\end{align*}
\end{prop}
\begin{proof} As $\nabla_{\b{e}_i}e_j(p)=[\b{e}_i,e_j]_{1,0}(p)=0$ and $$g(\nabla_{e_i}e_j,\b{e}_k)(p)=e_i(g_{j\b{k}})(p)-g(e_j,\nabla_{e_i}\b{e}_k)(p)=0$$ for any $i,j,k$, we get that $\Gamma_{\b{i}j}^k=\Gamma_{ij}^k=0$. To compute $\p\Gamma_{\b{i}j}^k$, we have $$\nabla_{e_l}\nabla_{\b{e}_i}e_j(p)=e_l(\Gamma_{\b{i}j}^k)(p)e_k+\Gamma_{\b{i}j}^k\nabla_{e_l}e_k(p)=e_l(\Gamma_{\b{i}j}^k)(p)e_k.$$ Also, as $T_\nabla(e_i,\b{e}_j)=2\sqrt{-1}\delta_{ij} \xi$,
\begin{align*}g(\nabla_{e_l}\nabla_{\b{e}_i}e_j, \b{e}_k)&=g(\nabla_{e_l}\nabla_{e_j}\b{e}_i,\b{e}_k)+g(\nabla_{e_l}[\b{e}_i,e_j]_H,\b{e}_k)\\&=g(\nabla_{e_l}[\b{e}_i,e_j]_H,\b{e}_k)=g(\nabla_{[\b{e}_i,e_j]_H}e_l,\b{e}_k)+g(T_\nabla(e_l,[\b{e}_i,e_j]_H),\b{e}_k).
\end{align*}
At $p$, we have $[\b{e}_i,e_j]_H=\nabla_{\b{e}_i}e_j-\nabla_{e_j}\b{e}_i=0$. From the above, we get that $\nabla_{e_l}\nabla_{\b{e}_i}e_j=0$ for any $i,j,l$. So $e_l(\Gamma_{\b{i}j}^k)(p)=0$ which gives that $\p\Gamma_{\b{i}j}^k=0$. Similarly, we also get that $\b{\p}\Gamma_{\b{i}j}^k=0$.

To compute the curvature, we have $\nabla_{e_i}\nabla_{\b{e}_j}e_k(p)=0$, $[e_i,\b{e}_j](p)=[e_i,\b{e}_j]_H(p)+g([e_i,\b{e}_j],\xi)\xi=-2\sqrt{-1}\delta_{ij}\xi$, and $\nabla_{\xi}e_i(p)=[\xi,e_i]_{1,0}(p)=0$. Then at $p$,
\begin{align*}
R_{i\b{j}k\b{l}}(p)&=g(\nabla_{e_i}\nabla_{\b{e}_j}e_k-\nabla_{\b{e}_j}\nabla_{e_i}e_k-\nabla_{[e_i,\b{e}_j]}e_k, \b{e}_l)(p)\\&=-g(\nabla_{\b{e}_j}\nabla_{e_i}e_k,\b{e}_l)(p)+2\sqrt{-1}\delta_{ij}g(\nabla_{\xi}e_k,\bar{e}_l)=-g(\nabla_{\b{e}_j}(\Gamma_{ik}^r e_r),\b{e}_l)(p)\\&=-\b{e}_j(\Gamma_{ik}^l)(p)-g(\Gamma_{ik}^p\nabla_{\b{e}_j}e_k,\b{e}_l)(p)=-\b{e}_j(\Gamma_{ik}^l)(p).
\end{align*}
Also, as $e_ig_{k\b{l}}=g(\nabla_{e_i}e_k,\b{e}_l)+g(e_k,\nabla_{e_i}\b{e}_l)=\Gamma_{ik}^rg_{r\b{l}}+\Gamma_{i\b{l}}^{\b{r}}g_{k\b{r}}$, then $$\b{e}_je_ig_{k\b{l}}=\b{e}_j(\Gamma_{ik}^r)g_{r\b{l}}+\Gamma_{ik}^r\b{e}_j(g_{r\b{l}})
+\b{e}_j(\Gamma_{i\b{l}}^{\b{r}})g_{k\b{r}}
+\Gamma_{i\b{l}}^{\b{r}}\b{e}_j(g_{k\b{r}}).$$ At $p$, as $\b{e}_j(g_{r\b{l}})=\b{e}_j(g_{k\b{r}})=0$ and $\b{e}_j(\Gamma_{i\b{l}}^{\b{r}})=\overline{e_j(\Gamma_{\b{i}l}^{r})}=0$, we get that $\b{e}_j(\Gamma_{ik}^l)(p)=\b{e}_je_i(g_{k\b{l}})(p)$. So $R_{i\b{j}k\b{l}}(p)=-\b{e}_je_i(g_{k\b{l}})(p).$
\end{proof}

When $m=2m_0$ and $(M,J,g)$ is an almost Hermitian manifold, the existence of the local normal quasi holomorphic frame and its properties have been established in \cite{Yu1} (Lemma 3.4, Lemma 4.1, Corollary 4.1 there).
\begin{prop}[Yu] \label{Yu}
Let $(M,J, g)$ be an $2m_0$-dimensional almost Hermitian manifold. Then for any $p\in M$, there exists a local (1,0)-frame $\{e_1, e_2, \cdots, e_{m_0}\}$ around $p$ such that $e_i$ is quasi holomorphic at $p$ and $g_{i\b{j}}(p)=\d_{i\b{j}}$, $dg_{i\b{j}}(p)=0$. With respect to the frame, the Christoffel symbols and the curvature components of the Chern connection satisfies: $$\Gamma_{ij}^k=\Gamma_{\b{i}j}^k=0,\ \p\Gamma_{\b{i}j}^k=0,
\ \ R_{i\b{j}k\b{l}}=-\b{e}_je_i(g_{k\b{l}}).$$
\end{prop}

In the rest of the section, we discuss properties of the partial horizontal exterior differential of forms on an almost CR manifold. Let $(M,HM,J)$ be a codimension-($m-2m_0$) almost CR manifold. By the existence of Riemannian metrics, there always exists a splitting $TM=HM\oplus F$, where $F$ is a ($m-2m_0$) subbundle in $TM$. Denote $T_{1,0}(M), T_{0,1}(M)$ the $i,-i$-eigenbundle of $J$ in $HM\otimes \B C$. Then $TM\otimes \B C=T_{1,0}(M)\oplus T_{0,1}(M)\oplus F\otimes \B C$. Let $T^{1,0}(M), T^{0,1}(M)$ be the dual spaces of $T_{1,0}(M), T_{0,1}(M)$. Denote $\wedge^k(M)=\wedge^k (T^*M\otimes \B C)$ and $\wedge^{p,q}(M)=\wedge^p(T^{1,0}(M))\otimes \wedge^q(T^{0,1}(M)$. Then we have
$$\wedge^k(M)=\oplus_{p+q+r=k} \wedge^{p,q}(M)\otimes \wedge^r (F^*\otimes \B C)$$ for any $k>0$. There is a natural real linear conjugate map from $\wedge^{p,q}(M)$ to $\wedge^{q,p}(M)$, mapping any $\vphi$ to $\b{\vphi}$. For $p+q=k$, denote $\pi^{p,q}: \wedge^k(M)\lra\wedge^{p,q}(M)$ the projection map. Given a $(p,q)$ form $\vphi\in \wedge^{p,q}(M)$, define
$$\p \vphi =\pi^{p+1,q}(d\phi), \ \ \ \bar{\p}\vphi=\pi^{p,q+1}(d\phi).$$
If $m-2m_0>0$, the $\p, \b{\p}$ operators would depend on the given splitting. Clearly, $\overline{\p \vphi}=\b{\p}\b{\vphi}$. From the definition, the Leibniz rule holds: $\p(g\vphi)=\p g\wedge \vphi+g\p\vphi, \bar{\p}(g\vphi)=\bar{\p}g\wedge \vphi+g\bar{\p}\vphi$ for smooth function $g$ on $M$.

We have
\begin{lem}\label{exter}
Let $(M,HM,J)$ and $(N,HN,J^N)$ be two almost CR manifolds and $f:M\lra N$ be an almost CR map. Assume that there are two splitting $TM=HM\oplus F_1$ and $TN=HN\oplus F_2$ such that $df(F_1)\subset F_2$. Then for any form $\vphi\in \wedge^{p,q}(N)$, we have $f^*\vphi\in \wedge^{p,q}(M)$. Also,  $$f^*(\p\vphi)=\p(f^*\vphi), \ \ \ f^*(\b{\p}\vphi)=\b{\p}(f^*\vphi).$$
\end{lem}
\begin{proof}
By definition, a form $\alpha$ is in $\wedge^{p,q}(M)$ if and only if the contraction $i_X(\alpha)=0$ for any $X\in F_1$. For $\vphi\in \wedge^{p,q}(N)$, as $df(F_1)\subset F_2$, we have $i_{df(X)}\vphi=0$. Then $i_X(f^*\vphi)=f^*(\iota_{df(X)}\vphi)=0$. So $f^*\vphi\in \wedge^{p,q}(M)$. For the differential, we have the classical relation $d(f^*\vphi)=f^*(d\vphi)$. As $f^*$ keeps the degrees, projecting the relation to $\wedge^{p+1,q}(M)$ and $\wedge^{p,q+1}(M)$ we get the above equalities.
\end{proof}

We also have the following maximum principle for a real function $u$.
\begin{lem}
Assume that $u$ is a real function on an almost CR manifold $(M,HM,J)$. Fix a splitting $TM=HM\oplus F$. If $u$ achieves a maximum at $p$, then $du(p)=0$ and $\p\b{\p}u(V,\b{V})(p)\leq 0$ for any $(1,0)$ vector field $V$ around $p$.
\end{lem}
\begin{proof}
By definition, we have $\p u(V)=V(u), \b{\p}u(V)=0, \b{\p}u(\b{V})=\b{V}(u)$. Then $$\p\b{\p}u(V,\b{V})=d\b{\p}u(V,\b{V})=V\b{V}(u)-[V,\b{V}]_{0,1}(u)$$
 At the maximum point $p$, if $(x_1,\cdots, x_m)$ is a coordinate chart of $M$, we have $du(p)=0$ and the matrix $(\f{\p^2 u}{\p x_i\p x_j})\leq 0$. So $YYu(p)\leq 0$ for any real vector field $Y$. Assume that $V=X-\sqrt{-1}JX$. At $p$, we get
\begin{align*}\p\b{\p}u(V,\b{V})&=XXu+JX(JX)u+\sqrt{-1}[X,JX]u-[V,\b{V}]_{0,1}(u)\\
&=XXu(p)+JX(JX)u(p)\leq 0
\end{align*}
\end{proof}

\noindent
We point out that $\p\b{\p}u(V,\b{V})$ is in general a complex function on $M$ if $m-2m_0>0$. When $m=2m_0$, i.e. $(M,J)$ is an almost complex manifold, direct calculation gives that $\p\b{\p}u=-\b{\p}\p u$. Then $\p\b{\p}u(V,\b{V})$ is a real function.

\section{Liouville type results}
In this section, we give the proofs of Theorem 1.1, 1.2, 1.3.

\subsection{Proof of Theorem 1.1}
 Let $(M, HM, J, \theta)$ be a $(2m_0+1)$ dimensional pseudo-Hermitian manifold and $(N,J^N,h)$ be a $2n_0$ dimensional almost Hermitian manifold. Assume that $f: M\lra N$ is a $(J,J^N)$ holomorphic map. Let $\xi$ be the Reeb vector field on $M$. By Definition \ref{def1} (i), $df(\xi)=0$ and $df(T_{1,0}(M))\subset T_{1,0}N$. Assume that $\{e_i, 1\leq i\leq m_0\}$ is any local $(1,0)$ frame on $M$, and $\{\t{e}_\a, 1\leq \a\leq n_0\}$ is any local $(1,0)$ frame on $N$ with $df(e_i)=f_i^\a \t{e}_{\a}$. Let $\theta^i$ and $\t{\theta}^\a$ be the coframes dual to $\{e_i\}$ and $\{\t{e}_\a\}$ on $M,N$. Then
\begin{align} \label{equa5}
f^*(\t{\theta}_\a)=f^\a_i\theta^i.\end{align}
The first structure equation of $M$ gives (\cite{DT}):
\begin{align}\label{equa9}
d\theta^i=\theta^j\wedge \theta_j^i+\theta\wedge A^i_{\b{j}}\b{\theta}^j,
\end{align}
where $A^i_{\b{j}}$ are the components of the pseudo-Hermitian torsion one form of $M$, $\theta_j^i=\Gamma_{kj}^i\theta^k+\Gamma_{\b{k}j}^i\b{\theta}^k,$ with $\Gamma_{kj}^i$ being the Christoffel symbols of the Tanaka-Webster connection.
The first structure equation of $N$ (\cite{CDRY}) gives
\begin{align} \label{equa8}
d\t{\theta}^\a=\t{\theta}^\beta\wedge \t{\theta}_\beta^\a+\frac{1}{2}T^\a_{\beta\g}\t{\theta}^\beta\wedge\t{\theta}^\g,
\end{align}
where $T^\a_{\beta\g}$ are the components of the torsion tensor of the Chern connection of  $N$ and $\t{\theta}^\a_\beta=\t{\Gamma}_{\gamma\beta}^\alpha\t{\theta}^\gamma+
\t{\Gamma}_{\b{\gamma}\beta}^\alpha\b{\t{\theta}}^\gamma$ with $\t{\Gamma}_{\gamma\beta}^\alpha$ being the Christoffel symbols. Taking exterior differential to (\ref{equa5}), we have:
\begin{align}
df^\a_i\wedge \theta^i+f^\a_i(\theta^j\wedge \theta_j^i+\theta\wedge A^i_{\b{j}}\b{\theta}^j)&=f^*(\t{\theta}^\beta\wedge \t{\theta}_\beta^\a+\frac{1}{2}T^\a_{\beta\g}\t{\theta}^\a\wedge\t{\theta}^\g) \notag \\
&=f^\beta_j\theta^j\wedge f^*\t{\theta}_\beta^\a+\frac{1}{2}f^*T^\a_{\beta\g}f^\a_i f^\g_j\theta^i\wedge \theta^j.  \label{equa6}
\end{align}
By comparing the $(1,1)$ form in equation (\ref{equa6}), we get the equation for $(0,1)$ form: \begin{align} \label{part11}
\b{\p} f^\a_i=f^{\a}_j\Gamma_{\b{k}i}^j\b{\theta}^k- f^{\beta}_i\t{\Gamma}^\a_{\b{\gamma}\beta}f^{\b{\g}}_{\b{j}}\b{\theta}^j.\end{align}
Also, comparing the components of $\theta\wedge \theta^i$, we get that
\begin{align} \label{xi} i_{\xi}df^\a_i=\xi(f^\a_i)=0.\end{align}
Taking partial exterior differential to (\ref{part11}) and applying Lemma \ref{exter}, we have  \begin{align}\label{part22}
\p\b{\p}  f^\a_i=\p\Gamma_{\b{k}i}^jf^{\a}_j\b{\theta}^k+\Gamma_{\b{k}i}^j\p(f^{\a}_j\b{\theta}^k)- f^*(\p\t{\Gamma}^\a_{\b{\gamma}\beta})f^{\beta}_if^{\b{\g}}_{\b{j}}\b{\theta}^j-
\t{\Gamma}^\a_{\b{\gamma}\beta}\p(f^{\beta}_if^{\b{\g}}_{\b{j}}\b{\theta}^j).
\end{align}

Next, we define and study a energy density function which was first introduced by Yang \cite{Yang1} on Hermitian manifolds. Denote $T_{1,0}(M)-\{0\}$ the space of nonzero $(1,0)$ vectors over $M$ which forms a $\mathbb C^{m_0}-\{0\}$ fiber bundle over $M$ with the projection map $\pi: T_{1,0}(M)-\{0\}\lra M$. Define the density function $\mathscr Y: T_{1,0}(M)-\{0\}\lra \mathbb R$ by $$\mathscr Y(W)=\dfrac{h(df(W),df(\bar{W}))}{g(W,\bar{W})},$$
where $W\in T_{1,0}(M)-\{0\}$. If the metrics $g$ and $h$ has local matrix expressions $(g_{i\b{j}}), (h_{\a\b{\beta}})$ and $W=W^ie_i$, then locally

$$ \mathscr Y(W)=\dfrac{h_{\a\bar{\beta}}f^{\a}_if^{\bar{\beta}}_{\bar{j}}W^i\overline{W^{j}}}{g_{k\bar{l}}W^{k}\overline{W^{l}}}.$$
From the definition, $f$ is constant if and only if $\mathscr Y\equiv 0$.

We will show by contradiction that under the conditions in Theorem 1.1, the function $\mathscr Y$ has to be zero. Denote $\mathbb PT_{1,0}(M)$ the complex projectivization of $T_{1,0}(M)$, which is a $\mathbb P^{m_0-1}$ fiber bundle over $M$. By the homogeneity, $\mathscr Y$ introduce a function on $\mathbb PT_{1,0}(M)$ which we still denote by $\mathscr Y$. As $M$ is compact, $\mathbb PT_{1,0}(M)$ is also compact. Suppose that $\mathscr Y$ is not constant zero. Then there exists a maximum point in $\mathbb PT_{1,0}(M)$ which corresponds to a maximum point $q$ of $\mathscr Y$ in $T_{1,0}(M)-\{0\}$, where $\mathscr Y$ is positive. Let $U$ be a local coordinate chart $p=\pi(q)$ so that with respect to a $(1,0)$ frame $\{e_1,\cdots,e_{m_0}\}$, $\pi^{-1}(U)\cong U\times (\mathbb C^{m_0}-\{0\})$. In the following lemma, we define a product CR structure on $\pi^{-1}(U)$ and derive its properties.

\begin{lem}\label{prod}
There exists a CR structure on $\pi^{-1}(U)$ satisfying: (i) the projection $\pi$ is a CR map; (ii) if $(W^1,W^2,\cdots,W^{m_0})$ are the local fiberwise complex coordinates, then on $\pi^{-1}(U)$,  $\b{\partial}W^i=\partial\ov{W}^i=0$ and $\partial\b{\partial}\overline{W^i}=\b{\partial}\partial W^i=0$.
\end{lem}
\begin{proof}
Since $\pi^{-1}(U)\cong U\times (\mathbb C^{m_0}-\{0\})$, the tangent bundle $T\pi^{-1}(U)\cong TU\oplus T(\mathbb C^{m_0}-\{0\})$. Under the isomorphism, define the Levi distribution on $\pi^{-1}(U)$ to be $HM|_U\oplus T(\mathbb C^{m_0}-\{0\})$ and the endomorphism $J_{T_{1,0}}=J\oplus J_{\mathbb C^{m_0}}$ where $J_{\mathbb C^{m_0}}$ is the standard integrable almost complex structure on $\mathbb C^{m_0}$. Obviously $\pi$ is a CR map from $\pi^{-1}(U)$ to $U$. The $(1,0)$ vector space for $J_{T_{1,0}}$ is spanned by $T_{1,0}(M)$ and $\{\frac{\partial}{\partial W^i},1\leq i\leq m_0\}$. As $\b{\partial}W^i|_{HM}=0$ and
$\b{\partial}W^i(\frac{\partial}{\partial \overline{W^j}})=\frac{\partial W^i}{\partial \overline{W^j}}=0$ for any $i,j$, we get $\b{\partial}W^i=0$. By conjugation, $\partial\ov{W}^i=0$

To show $\partial\b{\partial}\overline{W}^i=0$, we have $\partial\b{\partial}\ov{W}^i(v_1,v_2)=v_1(\b{\partial}\ov{W}^i(v_2))-v_2(\b{\partial}\ov{W}^i(v_1))-\b{\partial}\ov{W}^i([v_1,v_2]_{0,1})$ for $(1,0)$ vector $v_1$ and $(0,1)$ vector $v_2$. As $\b{\partial}\ov{W}^i|_{HM}=0$, $\b{\partial}\ov{W}^i(\frac{\partial}{\partial \overline{W_j}})=\delta^i_j$, $[\frac{\partial}{\partial W^i},TM]\subset TM, [\frac{\partial}{\partial \ov{W}^i},TM]\subset TM$, we get $\partial\b{\partial}\ov{W}^i(v_1,v_2)=0$ for $v_1,v_2\in TM$ or $v_1\in TM, v_2=\frac{\partial}{\partial \ov{W}^j}$. Finally, direct calculation gives $\partial\b{\partial}\ov{W}^i(\frac{\partial}{\partial W^j}, \frac{\partial}{\partial \ov{W}^k})=0$. So by linearity, $\partial\b{\partial}\ov{W}^i=0$. Taking conjugation, we get $\b{\partial}\partial W^i=0$
\end{proof}

Using the same notation as in \cite{Yang1} (see also \cite{Mas}), locally denote $$\mathscr F=h_{\a\bar{\beta}}f^{\a}_if^{\bar{\beta}}_{\bar{j}}W^i\overline{W^{j}}, \mathscr H=g_{k\bar{l}}W^{k}\overline{W^{l}}.$$ Equip $\pi^{-1}(U)$ with the product CR structure in Lemma \ref{prod}. We have
\begin{align}
\partial\bar{\partial} \mathscr Y=& -\mathscr Y(\partial\bar{\partial} \log \mathscr H)+\dfrac{\partial\bar{\partial}\mathscr F}{\mathscr H}\notag \\
 &-\dfrac{\partial \mathscr F\wedge \bar{\partial}\log \mathscr H+\partial \log \mathscr H\wedge \bar{\partial}\mathscr F}{\mathscr H}+\dfrac{\mathscr F\partial \log \mathscr H\wedge \bar{\partial}\log \mathscr H}{\mathscr H} \label{Y1}
 \end{align}

At the point $q$, the (1,0) tangent vector space of the above local product CR structure on $\pi^{-1}(U)$ is spanned by $T_{1,0}M(p)$ and $\{\frac{\partial}{\partial W_i},1\leq i\leq m_0\}$. Denote $q=(p,(V^1,V^2,\cdots, V^{m_0}))$.
Define the (1,0) vector $V\in T_q(\pi^{-1}(U))$ at $q$ by $V=V^1e_1+\cdots+V^{m_0}e_{m_0}$. We have the following

\begin{lem}\label{lem2} Let $(e_1, e_2, \cdots, e_{m_0})$ be a normal quasi holomorphic frame at $p$ and $(\tilde{e}_1, \t{e}_2, \cdots, \t{e}_{n_0})$ be a normal quasi holomorphic frame at $f(p)$. Then the following holds:

$$\partial\bar{\partial} \mathscr Y(V,\b{V})(q)\geq \dfrac{\mathscr Y}{\mathscr H}R_{i\b{j}k\b{l}}(p)V^i\b{V}^{j}V^k\b{V}^{l}-\dfrac{1}{\mathscr H}\tilde{R}_{\a\b{\beta}\g\b{\d}}(f(p))f^{\a}_if^{\b{\beta}}_{\b{j}}f^{\g}_{k}f^{\b{\d}}_{\b{l}}
V^i\b{V}^jV^k\b{V}^l,$$
where $R$, $\tilde{R}$ are the curvature components of $M$ and $N$.
\end{lem}
\begin{proof}
We evaluate equation (\ref{Y1}) on $(V,\b{V})$ at $q$. As the projection $\pi$ satisfies the conditions in Lemma \ref{exter}, we can actually evaluate each term on $M$ and $N$. By Proposition \ref{normal} at $p$, $\partial g_{i\b{j}}=\b{\partial}g_{i\b{j}}=0$.
As the vertical component of $V$ is zero, $\partial W^i(V)=\b{\partial}\ov{W}^i(\b{V})=\partial \mathscr H(V)=\b{\partial}\b{\mathscr H}(\b{V})=0$. So the terms in the second line of the equation (\ref{Y1}) vanish since they would involve $\partial \mathscr H, \b{\partial}\b{\mathscr H}$. For the first term of equation (\ref{Y1}), by Lemma \ref{exter} and Proposition \ref{normal},
$$\partial\bar{\partial} \log \mathscr H(V,\b{V})=\dfrac{\partial\bar{\partial} g_{i\b{j}}(e_k,\b{e}_l)V^i\b{V}^jV^k\b{V}^l}{\mathscr H}=-\dfrac{R_{i\b{j}k\b{l}}V^i\b{V}^jV^k\b{V}^l}{\mathscr H}.$$
To compute $\dfrac{\partial\bar{\partial}\mathscr F}{\mathscr H}(V,\b{V})$,
by Proposition \ref{normal} and Proposition \ref{Yu}, at $p$, $\Gamma_{\b{i}j}^k=\partial \Gamma_{\b{i}j}^k=0$; at $f(p)$, $\t{\Gamma}_{\b{\a}\beta}^\g=\partial{\t{\Gamma}}_{\b{\a}\beta}^\g=0$. So at $p$, equations (\ref{part11}),(\ref{part22}) give $\b{\p}f_i^{\a}=0, \p\b{\p}f_i^{\a}=0$. Next, we compute $\b{\p}\p f_i^{\a}$. For any $0\leq r,s\leq m_0$, we have
\begin{align*}
\p\b{\p}f_i^{\a}(e_r,\b{e}_s)+\b{\p}\p f_i^{\a}(e_r,\b{e}_s)&=e_r\b{e}_s(f_i^{\a})-[e_r,\b{e}_s]_{0,1}(f_i^{\a})-\b{e}_se_r(f_i^{\a})-[e_r,\b{e}_s]_{1,0}(f_i^{\a})\\
&=g([e_r,\b{e}_s],\xi)\xi(f_i^{\a})=0,
\end{align*}
where the last equality follows from (\ref{xi}). Therefore, at $p$, $\b{\p}\p f_i^{\a}=-\p\b{\p}f_i^{\a}=0$.
Then we have
\begin{align*}
\partial\bar{\partial}\mathscr F(V,\b{V})=&f^{\a}_iV^if^{\b{\beta}}_{\b{j}}\b{V}^j\p\b{\p}h_{\a\b{\beta}}(df(d\pi(V)),df( d\pi(\b{V})))+h_{\a\b{\beta}}(\p(f^{\a}_iW^i)\wedge \b{\p}(f^{\b{\beta}}_{\b{j}}\ov{W}^j))(V,\b{V})\\
\geq&f^{\a}_iV^if^{\b{\beta}}_{\b{j}}\b{V}^j\p\b{\p}h_{\a\b{\beta}}(df(d\pi(V)),df( d\pi(\b{V})))= -\tilde{R}_{\a\b{\beta}\g\b{\d}}f^{\a}_if^{\b{\beta}}_{\b{j}}f^{\g}_{k}f^{\b{\d}}_{\b{l}}
V^i\b{V}^jV^k\b{V}^l.
\end{align*}
The lemma is proved.
\end{proof}

Now we conclude the proof of Theorem 1.1. At the maximum point $q\in T_{1,0}(M)-\{0\}$, by Lemma 3.8, $\partial\bar{\partial} \mathscr Y(q)(V,\b{V})\leq 0$. As the pseudo-Hermitian holomorphic sectional curvature of $M$ is nonnegative (resp. positive) and the holomorphic sectional curvature of $N$ is negative (resp. nonpositive), by Lemma \ref{lem2}, $\partial\bar{\partial} \mathscr Y(q)(V,\b{V})>0$, which is a contradiction. Therefore, $\mathscr Y\equiv 0$ which gives that $df=0$. So $f$ is constant.

\subsection{Proof of Theorem 1.2} Let $(N,J^N,h)$ be a compact $2n_0$-dimensional almost Hermitian manifold and $(M, HM, J, \theta)$ be a $2m_0+1$-dimensional Sasakian manifold. Assume that  $f:N\lra M$ is a $(J^N,J)$ holomorphic map. By Definition \ref{def1}, $df_H(T_{1,0}N)\subset T_{1,0}M$, where $df_H=\pi_H\circ df$. Assume that $\{\t{e}_\a, 1\leq \a\leq n_0\}$ is any local $(1,0)$ frame on $N$ and $\{e_i, 1\leq i\leq m_0\}$ is any local $(1,0)$ frame on $M$. Then $df(\t{e}_{\a})=f_\a^i e_i+f_\a^0\xi$ and $df_H(\t{e}_{\a})=f_\a^i e_i$. Let $\t{\theta}^\a$ and $\theta^i$ be the coframes dual to $\{\t{e}_\a\}$ and $\{e_i\}$. Then
\begin{align} \label{equa7}
f^*(\theta^i)=f^i_\a \t{\theta}^\a. \end{align}
As $M$ is Sasakian, the first structure equation of $M$ is:
\begin{align}\label{equa10}
d\theta^i=\theta^j\wedge \theta_j^i.
\end{align}
where $\theta_j^i=\Gamma_{kj}^i\theta^k+\Gamma_{\b{k}j}^i\b{\theta}^k$. Taking exterior differential to (\ref{equa7}) and comparing the $(1,1)$ form components, we have \begin{align} \label{equa12}
\b{\p}f_\a^i=f_\beta^i\t{\Gamma}_{\b{\gamma}\a}^\beta\b{\t{\theta}}^\gamma-f_\a^j\Gamma_{\b{k}j}^if^{\b{k}}_{\b{\beta}}\b{\t{\theta}}^\beta.\end{align}
View $N$ as a codimension zero almost CR manifold. As $f$ is $(J^N,J)$ holomorphic, applying Lemma \ref{exter} we get, \begin{align}\p\b{\p}f_\a^{i}=\t{\Gamma}_{\b{\gamma}\a}^\beta\p(f_\beta^i\b{\t{\theta}}^\gamma)
+\p\t{\Gamma}_{\b{\gamma}\a}^\beta f_\beta^i\b{\t{\theta}}^\gamma-
\Gamma_{\b{k}j}^i\p(f_\a^jf^{\b{k}}_{\b{\beta}}\b{\t{\theta}}^\beta)-f^*(\p\Gamma_{\b{k}j}^i)f_\a^j
f^{\b{k}}_{\b{\beta}}\b{\t{\theta}}^\beta.\label{equa13} \end{align}

Denote $T_{1,0}(N)-\{0\}$ the space of nonzero $(1,0)$ vectors over $N$ which forms a $\mathbb C^{n_0}-\{0\}$ fiber bundle over $N$ with $\pi:T_{1,0}(N)-\{0\}\lra N$ being the projection. Let $W=W^\a\t{e}_{\a}$ be a nonzero $(1,0)$ vector of $N$. Define the density function $\mathscr Y: T_{1,0}(N)-\{0\}\lra \mathbb R$ by $$\mathscr Y(W)=\dfrac{g(df_H(W),df_H(\bar{W}))}{h(W,\bar{W})}=\dfrac{g_{i\bar{j}}f^{i}_{\a}f^{\bar{j}}_{\bar{\beta}}W^{\a}\overline{W^{\beta}}}{h_{\g\bar{\d}}W^{\g}\overline{W^{\d}}}.$$
By the definition, $\mathscr Y\equiv 0$ if and only if $df_H=0$, which means $f$ is horizontally constant. If $\mathscr Y$ is not identically zero. As $N$ is compact, by the homogeneity of $\mathscr Y$, there must be a maximum point $q\in T_{1,0}(N)-\{0\}$ such that $\mathscr Y(q)=\max \mathscr Y>0$. Let $p=\pi(q)\in N$. Then there exists a local coordinate chart $U\subset N$ of $p$ so that $\pi^{-1}(U)\cong U\times (\mathbb C^{n_0}-\{0\})$. The following lemma is similar to Lemma \ref{prod}.
\begin{lem}\label{prod2}
There exists an almost structure on $\pi^{-1}(U)$ satisfying: (i) the projection $\pi$ is a pseudoholomorphic map; (ii) if $(W^1,W^2,\cdots,W^{n_0})$ are the local fiberwise complex coordinates, then on $\pi^{-1}(U)$,  $\b{\partial}W^\a=\partial\ov{W}^\a=0$ and $\partial\b{\partial}\overline{W^\a}=\b{\partial}\partial W^\a=0$.
\end{lem}
\begin{proof} As the tangent bundle $T\pi^{-1}(U)\cong TU\oplus T(\mathbb C^{n_0}-\{0\})$, we define the almost complex structure on $\pi^{-1}(U)$ by $J_{T_{1,0}}=J^N\oplus J_{\mathbb C^{n_0}}$. Then $\pi$ is pseudoholomorphic. The proof of property (ii) is the same with that in  Lemma \ref{prod}.
\end{proof}

The (1,0) tangent vector space of the above almost structure on $\pi^{-1}(U)$ is spanned by $T_{1,0}N$ and $\{\frac{\partial}{\partial W^\a},1\leq \a\leq n_0\}$. (The technique and Lemma \ref{prod2} may be needed in the arguments in \cite{Mas}).

Now we can prove Theorem 1.2. Assume $q=(p,(V^1,V^2,\cdots, V^{m_0}))$ in local coordinates. Define a (1,0) vector $V\in T_q(\pi^{-1}(U))$ by $V=V^1\t{e}_1+\cdots+V^{n_0}\t{e}_{n_0}$. Applying the maximum principle in Lemma 3.8 where $F=0$, we have $\partial\bar{\partial} \mathscr Y(q)(V,\b{V})\leq 0$.

On the other hand, denote $\mathscr F=g_{i\bar{j}}f^{i}_{\a}f^{\bar{j}}_{\bar{\beta}}W^{\a}\overline{W^{\beta}}, \mathscr H=h_{\g\bar{\d}}W^{\g}\overline{W^{\d}}.$ Equip $\pi^{-1}(U)$ with the product CR structure in Lemma \ref{prod}. We get\begin{align} \label{total}
\partial\bar{\partial} \mathscr Y=-\mathscr Y(\partial\bar{\partial} \log \mathscr H)+\frac{\partial\bar{\partial}\mathscr F}{\mathscr H}
 -\frac{\partial \mathscr F\wedge \bar{\partial}\log \mathscr H+\partial \log \mathscr H\wedge \bar{\partial}\mathscr F}{\mathscr H}+\frac{\mathscr F\partial \log \mathscr H\wedge \bar{\partial}\log \mathscr H}{\mathscr H}.\end{align}
Then we choose $\{\t{e}_\a\}$ and $\{e_i\}$ to be local normal quasi holomorphic frames at $p$ and $f(p)$ respectively. Evaluating to $(V,\b{V})$, the last two terms vanish as before. For the first  term, by Lemma \ref{prod2} and Proposition \ref{Yu},
\begin{align}\label{first}
\partial\bar{\partial} \log \mathscr H(V,\b{V})=-\dfrac{\t{R}_{\a\b{\beta}\gamma\b{\d}}V^\a\b{V}^\beta V^\g\b{V}^\d}{\mathscr H},\end{align}
where $\t{R}$ denotes the curvature tensor of the $N$. For the second term, by equations (\ref{equa12}), (\ref{equa13}), at $p$, $\b{\p}f_\a^i=0, \p\b{\p}f_\a^{i}=0$. Also, on an almost complex manifold, $\b{\p}\p f_\a^{i}=-\p\b{\p}f_\a^{i}=0$. Then the same calculation as in Theorem 1.1 gives that \begin{align}\label{second}
\partial\bar{\partial}\mathscr F(V,\b{V})\geq -R_{i\b{j}k\b{l}}f^{i}_\a f^{\b{j}}_{\b{\beta}}f^{k}_{\g}f^{\b{l}}_{\b{\d}}
V^\a\b{V}^\beta V^\g\b{V}^\d.
\end{align}
Combing (\ref{total}), (\ref{first}) and (\ref{second}), as the holomorphic sectional curvature of $N$ is nonnegative (resp. positive) and the pseud-Hermitian sectional curvature of $M$ is negative (resp. nonpositive), we get $\partial\bar{\partial} \mathscr Y(q)(V,\b{V})>0$ which is a contradiction. Therefore, $\mathscr Y\equiv 0$ and $f$ is horizontally constant.

\subsection{Proof of Theorem 1.3} Let $(M, HM, J, \theta)$ be a $2m_0+1$-dimensional pseudo-Hermitian manifold and $(N,HN, J^N, \theta^N)$ be a $2n_0+1$-dimensional Sasakian manifold. Let $f:M\lra N$ be a transversally holomorphic map. By Definition \ref{def2}, this means $df(L)\subset df(L^N)$ and $df_{M,N}(T_{1,0}M)\subset T_{1,0}N$, where $df_{M,N}=\pi_N\circ df\circ i_{HM}$.  Assume that $\{e_i, 1\leq i\leq m_0\}$ is any local $(1,0)$ frame on $M$ around $p\in M$, $\{\t{e}_\a, 1\leq \a\leq n_0\}$ is any local $(1,0)$ frame around $f(p)\in N$ and $\xi, \t{\xi}$ are the Reeb vector fields on $M$, $N$. Then $df(e_i)=f_i^\a \t{e}_\a+f_i^0\t{\xi}$ and $df(\xi)=f_0^0\t{\xi}$. So $df_{M,N}(e_i)=f_i^\a \t{e}_\a$. Let $\theta^i$ and $\t{\theta}^\a$ be the coframes dual to $\{e_i\}$ and $\{\t{e}_\a\}$ on $M,N$. Then we have
\begin{align} \label{equa17}f^*(\t{\theta}^\a)=f^\a_i\theta^i. \end{align}
The first structure equation of $M$ gives that
\begin{align}
d\theta^i&=\theta^j\wedge \theta_j^i+\theta\wedge A^i_{\b{j}}\b{\theta}^j,
\end{align} where $\theta_j^i=\Gamma_{kj}^i\theta^k+\Gamma_{\b{k}j}^i\b{\theta}^k$. Since $N$ is Sasakian, the first structure equation of $N$ gives
\begin{align}
d\t{\theta}^\a&=\t{\theta}^\beta\wedge \t{\theta}_\beta^\a,
\end{align} where $\t{\theta}^\a_\beta=\t{\Gamma}_{\gamma\beta}^\alpha\t{\theta}^\gamma+
\t{\Gamma}_{\b{\gamma}\beta}^\alpha\b{\t{\theta}}^\gamma.$  Taking exterior differential to (\ref{equa17}) to get
\begin{align} \label{df33}
df^\a_i\wedge \theta^i+f^\a_i(\theta^j\wedge \theta_j^i+\theta\wedge A^i_{\b{j}}\b{\theta}^j)=f^*(\t{\theta}^\beta\wedge \t{\theta}_\beta^\a)=f^\beta_j\theta^j\wedge f^*( \t{\theta}_\beta^\a).
\end{align}
Comparing the $(1,1)$ form in (\ref{df33}), we have \begin{align} \label{part1}
\b{\p} f^\a_i=f^{\a}_j\Gamma_{\b{k}i}^j\b{\theta}^k- f^{\beta}_i\t{\Gamma}^\a_{\b{\gamma}\beta}
f^{\b{\g}}_{\b{j}}\b{\theta}^j.\end{align}
Comparing the components of $\theta\wedge \theta^i$ in (\ref{df33}), we also get
\begin{align} \label{part111} \xi(f^\a_i)=0
\end{align}
As $\pi$ satisfies the conditions of Lemma \ref{exter} and $f$ is transversally holomorphic, from (\ref{part1}) we have \begin{align}\label{part2}
\p\b{\p}  f^\a_i=\p\Gamma_{\b{k}i}^jf^{\a}_j\b{\theta}^k+\Gamma_{\b{k}i}^j\p(f^{\a}_j\b{\theta}^k)- f^*(\p\t{\Gamma}^\a_{\b{\gamma}\beta})f^{\beta}_if^{\b{\g}}_{\b{j}}\b{\theta}^j-
\t{\Gamma}^\a_{\b{\gamma}\beta}\p(f^{\beta}_if^{\b{\g}}_{\b{j}}\b{\theta}^j).
\end{align}
Let $W=W^ie_i\in T_{1,0}M$. Assume the Webster metric on $N$ is $h$. Define the density function $\mathscr Y: T_{1,0}(M)-\{0\}\lra \mathbb R$ by $$\mathscr Y(W)=\dfrac{h(df_{M,N}(W),df_{M,N}(\bar{W}))}{g(W,\bar{W})}=\dfrac{h_{\a\bar{\beta}}f^{\a}_if^{\bar{\beta}}_{\bar{j}}W^i\overline{W^{j}}}{g_{k\bar{l}}W^{k}\overline{W^{l}}}.$$
By the definition, $\mathscr Y\equiv 0$ if and only if $df_{M,N}=0$, which means $f$ is horizontally constant. We show that $Y\equiv 0$ under the conditions of Theorem 1.3. As before, if $\mathscr Y$ is not identically zero, since $M$ is compact, there must be a maximum point $q\in T_{1,0}(M)-\{0\}$ such that $\mathscr Y(q)=\max \mathscr Y>0$. Let  $U$ be a neighborhood of $q$ so that $\pi^{-1}(U)\cong U\times (\mathbb C^{m_0}-\{0\})$ and $q=(p,(V^1,V^2,\cdots, V^{m_0}))$ in this coordinates. Equip $\pi^{-1}(U)$ with the CR structure in Lemma \ref{prod}. Define a (1,0) vector $V\in T_q(\pi^{-1}(U))$ by $V=V^1\t{e}_1+\cdots+V^{n_0}\t{e}_{n_0}$. Applying the maximum principle in Lemma 3.8, we have $\partial\bar{\partial} \mathscr Y(q)(V,\b{V})\leq 0$.

Next, let $\{e_i\}$ and $\{\t{e}_\a\}$ be local normal quasi pseudoholomorphic frames around $p$ and $f(p)$. Combing with (\ref{part1}), (\ref{part111}), (\ref{part2}), the same argument as in Theorem 1.1 gives that $$\partial\bar{\partial} \mathscr Y(q)(V,\b{V})\geq \dfrac{\mathscr Y}{\mathscr H}R_{i\b{j}k\b{l}}V^i\b{V}^{j}V^k\b{V}^{l}-\dfrac{1}{\mathscr H}\tilde{R}_{\a\b{\beta}\g\b{\d}}f^{\a}_if^{\b{\beta}}_{\b{j}}f^{\g}_{k}f^{\b{\d}}_{\b{l}}
V^i\b{V}^jV^k\b{V}^l.$$ By the curvature conditions of $M$ and $N$, we have $\partial\bar{\partial} \mathscr Y(q)(V,\b{V})>0$ which is a contradiction. Therefore, $\mathscr Y\equiv 0$ and $f$ is horizontally constant.

\section{Almost CR structures on complex vector bundles over an almost CR manifold}

In this section, we construct global almost CR structures on a complex vector bundles over an almost CR manifold. Assume that $M$ is a $m$-dimensional manifold with a codimension $m-2m_0$ almost CR structure $(HM,J)$. We derive

\begin{prop}\label{5.1} Let $p: E\lra M$ be a $k$-dimensional complex vector bundle over an almost CR manifold $(M, HM, J)$. Then associated to any linear connection on $E$, there is a codimension $m-2m_0$ almost CR structure $(HE, J_E)$ on $E$ such that the map $p$ is an almost CR map.
\end{prop}

\begin{proof}
As a complex vector bundle, $E$ can be also viewed as a $2k$-dimensional real vector bundle with a bundle map $I:E\lra E$ such that $I^2=-id$. Denote the tangent bundle of $E$ by $TE$. The vertical bundle $VE$ is the kernel of $dp$ in $TE$. $VE$ is canonically isomorphic to the pull back bundle $p^* E$ on $E$. Therefore, the bundle map $I$ pull back to an endomorphism $J_I:VE\lra VE$ such that $J_I^2=-id$.

Next, recall that a (real) linear connection on $E$ is a map $\nabla^E: \Gamma(M,E)\lra \Gamma(M, T^*M\otimes E)$ such that $\nabla^E (fs)=f\nabla^E s+df\otimes s$ where $s$ is smooth section of $E$ and $f$ is a smooth function on $M$. The existence of linear connections is well known (see \cite{KN}). Fix a linear connection $\nabla^E$. Denote $P(E)$ the principal $GL(2k, \mathbb R)$ bundle determined by $E$ and $\pi: P(E)\lra M$ the projection. Let $\tilde{\pi}$ be the induced projection $\tilde\pi: P(E)\times \mathbb R^{2k}\lra M$ where the map on the second factor is trivial. There is a natural quotient map $q: P(E)\times \mathbb R^{2k}\lra E$ such that $p\circ q=\tilde\pi$. By the theory of connections on principal bundles (\cite{KN}), $\nabla_E$ determines a connection one-form on $P(E)$, which corresponds to a $GL(2k, \mathbb R)$-equivariant horizontal distribution $\t{H}P$ in $TP(E)$. Here a horizontal distribution means a $m$-dimensional distribution such that $d\pi_w|_{(\t{H}P)_w}: (\t{H}P)_w\lra T_{\pi(x)}M$ is an isomorphism for each $w\in P(E)$. The distribution $\t{H}P$ maps into $P(E)\times \mathbb R^{2k}$ as $\t{H}P\times 0$ which passes to a distribution $\t{H}E=dq(\t{H}P\times 0)$ over $E$. As $d\tilde\pi=dp\circ dq$, at each point $v\in E$, the projection $dp|_{\t{H}E}: (\t{H}E)_v\lra T_{p(v)}M$ is an isomorphism. Therefore, $\t{H}E$ is a horizontal distribution of $E$. As $dp(VE)=0$, we have $(VE)_v\cap (\t{H}E)_v=0$ at each $v\in E$. So $TE=\t{H}E\oplus VE$.

With the horizontal distribution $\t{H}E$, we pull back $HM$ through $dp$ and define a Levi distribution $HE$ of $E$. Indeed, denote $dp_{\t{H}E}$ the restriction of $dp$ to $HrE$, which is an isomorphism at each fiber $(\t{H}E)_v$. Let $HE=dp_{\t{H}E}^{-1}(HM)\oplus VE$. It is a $2m_0+2k$ subbundle of $TE$. Then define $J_E: HE\lra HE$ by $$J_E|_{dp_{\t{H}E}^{-1}(HM)}=dp_{\t{H}E}^{-1}\circ J\circ dp_{\t{H}E};\ \ \ \  J_E|_{VE}=J_I.$$
Then $J_L^2=-id$. So $(HE, J_E)$ forms a codimension $(m-2k)$ almost CR structure on $E$. As $dp(HE)=HM$, $dp\circ J_E=J\circ dp$, the projection $p$ is an almost CR map.
\end{proof}

We can describe $\t{H}E$ and $VE$ explicitly using local frames as follows. Assume that $U$ is an open chart of $M$ with coordinate $(x^1, \cdots, x^m)$ and $\{s_j\}, 1\leq j\leq 2k$ is a local frame of $E$ on $U$. With respect to $\{s_j\}$, $p^{-1}(U)\cong U\times \mathbb R^{2k}$ with coordinates $(x^1, x^2, \cdots, x^m, y^1, \cdots, y^{2k})$. Then $\{\f{\p}{\p x^i}, \f{\p}{\p y^j}, 1\leq i\leq m, 1\leq j\leq 2k \}$ forms a local frame of $TE$. Assume that $\nabla s_j=\om_j^p s_p$, where $\om_j^p$ are real 1-forms on $M$. Then at $v=(a^1, \cdots, a^m, b^1,\cdots, b^{2k})$, we can express $(VE)_v$ and $(\t{H}E)_v$ as:
\begin{align*}
(VE)_v=span\{\f{\p}{\p y^j}\}, 1\leq j\leq 2k,\\
(\t{H}E)_v=span\{\f{\p}{\p x^i}-b^{j}\om_j^p(\f{\p}{\p x^i})\f{\p}{\p y^p}\}, 1\leq i\leq m.
 \end{align*}
 By the expressions, $\t{H}E\cong p^*TM, VE\cong p^*E$.

To describe the local expression of the almost CR structure $(HE,J_E)$ on $E$, assume that $\{v_i, 1\leq i\leq 2m_0\}$ is a local frame of $HM$. Then at $v=(a^1, \cdots, a^m, b^1,\cdots, b^{2k})$, $dp_{\t{H}E}^{-1}(v_i)=v_i-b^{j}\om_j^p(v_i)\f{\p}{\p y^p}$. Here we identify $HM$ to be in the first factor of $TE\cong TM\times T\mathbb R^{2k}$. For the complex structure $I$ on $E$, let $I(s_j)=I_j^l s_l$. Then the Levi distribution $HE$ is spanned by $\{dp_{\t{H}E}^{-1}(v_i), \f{\p}{\p y^j}, 1\leq i\leq 2m_0, 1\leq j\leq 2k\}$. The bundle map is given by $$J_E(dp_{\t{H}E}^{-1}(v_i))=dp_{\t{H}E}^{-1}(J(v_i)),\ \ \ \  J_E(\f{\p}{\p y^j})=I_j^l\f{\p}{\p y^l}.$$
 A basis of $T_{1,0}M$ is given by $\{e_i=\dfrac{1}{2}(v_i-\sqrt{-1}J(v_i))\}$. Denote $$\hat{e}_i=\dfrac{1}{2}(dp_{\t{H}E}^{-1}(v_i)-\sqrt{-1}J_E(dp_{\t{H}E}^{-1}(v_i))=dp_{\t{H}E}^{-1}e_i,\ \ \ \ f_j=\dfrac{1}{2}(\f{\p}{\p y^j}-\sqrt{-1}I(\f{\p}{\p y^j})),$$ then $\{\hat{e}_i, f_j\}$ gives a basis of $T_{1,0}(E)$ for the almost CR structure $(HE,J_E)$.

For an almost CR structure $(HM, J)$, the endomorphism $J$ makes $HM$ to be a complex vector bundle. From Proposition \ref{5.1} we have
\begin{cor} \label{5.2} Let $(HM, J)$ be a codimension-$d$ almost CR structure on $M$. Then the space $HM$ admits codimension-$d$ almost CR structures such that the projection $p: HM\lra M$ is an almost CR map.
\end{cor}
Denote the complex projectivization of $HM$ to be $\mathbb PHM=(HM-\{0\})/\sim$, where $\{0\}$ denotes the zero section and $\sim$ is the equivalent relation that $v_1\sim v_2$ if and only if $v_1=(a+bJ)v_2, a,b\in \mathbb R$. With the canonical action of $GL(n,\mathbb C)$ on $\mathbb CP^{n-1}$, $\mathbb PHM$ is a $\mathbb CP^{n-1}$ fiber bundle associated to the principal $GL(n,\mathbb C)$-bundle determined by $HM$. Then a connection one-form on the principal $GL(n,\mathbb C)$-bundle determines a horizontal distribution which passes to be a horizontal distribution on $\mathbb PHM$. On the other side, as a complex quotient bundle of a complex vector bundle, the vertical distribution of $\mathbb PHM$ holds a natural endomorphism $J_V$ with $J_V^2=-id$. Pulling back $HM$ and $J$ to the horizontal distribution on $\mathbb PHM$ and together with $J_V$, we obtain an almost CR structure. Altogether we have
\begin{cor} Let $(HM, J)$ be a codimension-$d$ almost CR structure on $M$. Then the complex projectivization $\mathbb PHM$ admits codimension-$d$ almost CR structures such that the projection $Pr: \mathbb PHM\lra M$ is an almost CR map.
\end{cor}

Now assume that the almost CR structure $(HM, J)$ is integrable. With Corollary \ref{5.2}, it is natural to ask the following.
\begin{que}
Does there exist a CR structure on $HM$ such that the projection is a CR map?
\end{que}
At this point, we are not able to fully answer the above question. It may be related to the existence of some special linear connections. When we consider the tangent bundle $TM$ of a CR manifold $M$, we are able to construct a codimension-$2d$ CR structure on $TM$. Our method is based on techniques in Yano-Ishihara \cite{YI}.

\begin{thm}
Let $(M,HM,J)$ be a codimension-$d$ CR manifold. Then there exists a codimension-$2d$ CR structure on $TM$ such that the projection is a CR map.
\end{thm}

\begin{proof}
The construction relies on the vertical lift and complete lift of vector fields from $M$ to $TM$ introduced in \cite{YI} which we briefly discuss below. Assume that $U$ is a local chart of $M$ with coordinates $(x^1, \cdots, x^m)$ so that $\{\f{\p}{\p x^1},\cdots,\f{\p}{\p x^m}\}$ consists of a local frame of $TM$. Denote $(y^1,\cdots, y^m)$ to be the vector coordinates with respect to the frame. Then $(x^1, \cdots, x^m, y^1,\cdots,y^m)$ gives a local coordinates of $TU$. Let $\{\f{\p}{\p X^1}, \cdots, \f{\p}{\p X^m}, \f{\p}{\p Y^1}, \cdots, \f{\p}{\p Y^m}\}$ be the corresponding local frame of the tangent bundle $T(TM)$. Assume that $Z$ is a vector field of $M$ locally given by $Z=z^i\f{\p}{\p x^i}$. The vertical lift of $Z$ is a vector field $Z^V$ on $TM$ locally defined by $Z^V=z^i\f{\p}{\p Y^i}.$ The complete lift of $Z$ is a vector field on $TM$ defined by $Z^C=z^i\f{\p}{\p X^i}+y^j\f{\p z^i}{\p x^j}\f{\p}{\p Y^i}.$ For a $(1,1)$ tensor $A=A_i^j\f{\p}{\p x^j}\otimes dx^i$ on $M$, the complete lift of $A$ is defined to be a $(1,1)$ tensor on $TM$ by $$A^C=A_i^j\f{\p}{\p X^j}\otimes dX^i+y^k\f{\p A_i^j}{\p x^k}\f{\p}{\p Y^j}\otimes dX^i+A_i^j\f{\p}{\p Y^j}\otimes dY^i.$$ It is straightforward to verify that the above constructions do not depend on local coordinates and are globally defined. Let $Z,W$ be two vector fields on $M$ and $A$ be a $(1,1)$ tensor. From the local definitions, the following identities hold (see Proposition 3.3, 3.9 in \cite{YI})
\begin{align}
A^CZ^C=(AZ)^C, A^CZ^V=(AZ)^V, (A^2)^C=(A^C)^2,\notag\\
 [Z^V,W^V]=0, [Z^V,W^C]=[Z,W]^V, [Z^C,W^C]=[Z,W]^C. \label{com}
 \end{align}
Assume that $(HM, J)$ is a codimension-$d$ CR structure on $M$. The vertical lift and complete lift of $HM$ to $TM$ give two distributions on $TM$ which are denoted by $(HM)^V$ and $(HM)^C$. From the definition, we have $(HM)^V\cap (HM)^C=0$. Choose a splitting of $TM=HM\oplus F$ where $F$ is the complement distribution of $HM$ and extend $J$ to $TM$ by requiring $J|_F=0$. Consider the complete lift of $J$ which is $J^C$. From the above identities we get that the restriction of $J^C$ to $(HM)^V\oplus (HM)^C$ does not depend on the choice of $F$. Also $(J^C)^2=-id$ on $(HM)^V\oplus (HM)^C$. So $((HM)^V\oplus (HM)^C, J^C)$ forms an codimension-2d almost CR structure on $TM$.

To prove that $J^C$ is integrable, first notice that $[J^CZ^C,W^C]+[Z^C,J^CW^C]=([JZ,W]+[Z,JW])^C$ which lies in $(HM)^V\oplus (HM)^C$. The same holds for $Z^V, W^V$, etc. So the first integrability condition of a CR structure holds. Denote the Nijenhuis tensor of $J$ by $N_J(Z,W)=[JZ,JW]-J[JZ,W]-J[Z,JW]-[Z,W]$. From the identities (\ref{com}), we have $N_{J^C}(Z^V,W^V)=0$, $N_{J^C}(Z^C,W^C)=(N_J(Z,W))^C$ and $N_{J^C}(Z^C,W^V)=(N_J(Z,W))^V$. As $J$ is integrable, $N_J=0$. So we have $N_{J^C}=0$. Therefore, $((HM)^V\oplus (HM)^C, J^C)$ is integrable.
\end{proof}

\begin{remk}
When $d=0$, the above CR structure is just the canonical complex structure on the tangent bundle of a complex manifold (\cite{Kru}).
\end{remk}

\end{document}